\documentclass[oneside,reqno]{amsart}
\usepackage{amssymb,amsmath}
\usepackage{amsfonts,amsthm}
\usepackage{xcolor}
\usepackage{amssymb}
\usepackage{amsfonts}
\usepackage[english]{babel}
\usepackage[utf8]{inputenc} 
\usepackage{csquotes}

\usepackage{bbm}
\usepackage{hyperref}
\hypersetup{
    colorlinks=true,
    linkcolor=blue,
    filecolor=magenta,      
    urlcolor=cyan,
} 

\usepackage[all]{xy}
\usepackage{amsfonts,amsmath,verbatim,amssymb,amscd,latexsym}
\usepackage{amsthm,calc,enumerate, amscd}
\usepackage{graphicx}
\usepackage[T1]{fontenc}
\usepackage{pdfpages}
\usepackage{mathtools}

\newtheorem{theorem}{Theorem}[section]
\newtheorem{proposition}{Proposition}[section]
\newtheorem{lemma}[theorem]{Lemma}

\newtheorem*{thm1.1}{\bf Theorem 1.1}
\newtheorem*{thm1.2}{\bf Theorem 1.2 (Unconditional)}
\newtheorem*{lem5.3}{\bf Lemma 5.3}

\theoremstyle{definition}

\newtheorem{remark}[]{{ \bf Remark:}}

\numberwithin{equation}{section}
\begin{document}
\title{Subconvexity bound for  $GL(2)$ L-functions: \lowercase{t}-aspect} 
\author{Ratnadeep Acharya, Sumit Kumar, Gopal Maiti and Saurabh Kumar Singh}

\address{ Stat-Math Unit,
Indian Statistical Institute, 
203 BT Road,  Kolkata-700108, INDIA.}

\email{ratnadeepacharya87@gmail.com}
\email{sumitve95@gmail.com}
\email{g.gopaltamluk@gmail.com}

\email{skumar.bhu12@gmail.com}

\date{\today}


\begin{abstract}
Let $f $  be a holomorphic Hecke eigenform  or a Hecke-Maass cusp form for the full modular group $ SL(2, \mathbb{Z})$. In this paper we shall use circle method to prove the Weyl exponent for $GL(2)$ $L$-functions. We shall prove that 

\[
 L \left( \frac{1}{2} + it, f \right) \ll_{f, \epsilon} \left( 2 + |t|\right)^{1/3 + \epsilon}, 
\]  for any $\epsilon > 0.$
\end{abstract}
\maketitle 

\section{ Introduction }

Estimating the central values of $L$-functions is one of the most important problems in number theory. In this paper we shall deal with the $t$-aspect of sub-convexity bound for $GL(2)$ $L$-functions. Let $f $  be a holomorphic Hecke eigenform, or a Maass cusp form for the full modular group $ SL(2, \mathbb{Z})$ with normalised Fourier coefficient $\lambda_f(n)$. The $L$-series associated with  $f$ is given by
\[
L(s, f)= \sum_{n=1}^\infty \frac{\lambda_f(n)}{n^s} \ =  \prod_p \left( 1 -\lambda_f(p) p^{-s} + p^{-2s} \right)^{-1} \ \ \ (\Re s>1).
\]
It has been proved that the series $ L(s, f)$ extends to an entire function and satisfies a functional equation relating $s$ with $1-s$. The convexity problem in $t$-aspect deals with the size of $L(s, f)$ at the central line $\Re s = 1/2$.  The functional equation together with the Phragm{\' e}n–Lindel{\"o}f principle  and asymptotic of the Gamma functions  gives us the convexity bound, or the trivial bound,  $L(1/2+ it, f)\ll t^{1/2+ \epsilon}$. The sub-convexity problem is to obtain a bound of the form $L(1/2+ it, f) \ll t^{1/2 -\delta},$ for any $\delta>0.$ In this paper we shall prove the following theorem:

\begin{theorem} \label{main thm}
Let $f $  be either a holomorphic Hecke eigenform or a Maass cusp form for the full modular group $ SL(2, \mathbb{Z})$. On the central line $\sigma= 1/2$, we have the following Weyl bound
\[
L\left( \frac{1}{2} + it, f \right) \ll (|t|+2)^{ 1/3 +\epsilon} ,
\] for any $\epsilon >0$. 
\end{theorem}
\begin{remark}
 The method of the proof also works for any congruence subgroup $\Gamma_0(N)$, where $N$ is any natural number (not necessarily square free).
\end{remark}

 Let us  briefly recall the history of the $t$-aspect sub-convexity bound for $L$-functions.  The convexity bound for   the Riemann zeta function is given by 
\begin{equation} \label{conv for zeta}
\zeta \left( \frac{1}{2} + it \right) \ll t^{1/4 + \epsilon},\ \ \ \  (\epsilon> 0).
\end{equation}

  Lindel{\" o}f hypothesis asserts that the exponent $1/4 + \epsilon$ can be replaced by $\epsilon$. Sub-convexity bound for $\zeta(s)$ was first proved by G. H. Hardy and J. E. Littlewood, and H. Weyl independently.

It was first written down by E. Landau in a slightly refined form, and has been generalised to all Dirichlet $L$-functions. Since then it is a very hot topic for research. Many eminent mathematicians have worked on it and improved the exponent in \eqref{conv for zeta}. The latest bound is due to  J. Bourgain who proves the exponent $13/84$.

 The $t$-aspect  Weyl exponent  for  $GL(2)$ $L$-functions is expected to be $1/3.$ For holomorphic forms,   this was first proved by A. Good \cite{GOOD} using the spectral theory of automorphic functions. M. Jutila \cite{MJ} has given an alternative proof   based only on the functional properties of $L(s, f)$ and $L(s, f\otimes \chi)$, where $\chi$ is an additive character. The arguments used in his proof were flexible enough to be adopted for  the Maass cusp forms, as shown by   Meurman \cite{MERU1}, who proved the result for Maass cusp forms.  A. Good mean value estimate  itself was extended by  M. Jutila \cite{MJ1}  to prove the Weyl bound for Maass cusp forms, yet in another way. Using Kloosterman's circle method and conductor lowering trick introduced by R.Munshi, Aggarwal and Singh \cite{AS} proved the Weyl bound for  $GL(2)$ $L$-functions.

The aim of this paper is to use  $GL(2)$ circle method  to prove the Weyl bound for $GL(2)$ $L$-functions. This is the first instance where $GL(2)$ circle method is being used to obtain the Weyl bound. We carry out the suggestions of R.Munshi in this paper. We introduce one more layer in this technique by summing over the weights. This paper serves as a precursor to an upcoming paper of R.Munshi.

\section{Sketch of the proof}
To prove our theorem, we start with the following Fourier sum:
\begin{align}\label{Fourier sum}
 \mathcal{F} &=\sum_{k\sim K} W\left(\frac{k-1}{K}\right) \sideset{}{^\dagger}\sum_{ \psi (\textrm{mod}\ q)}\sum_{f\in H_k(q,\Psi)}\omega_f^{-1}\mathop{\sum \sum}_{m, \ell=1}^\infty \lambda_f(m)\lambda_F(m)\psi(\ell)  U \left(\frac{m \ell^2}{N} \right) \nonumber \\
 & \hspace{3cm}\times \sum_{n=1}^\infty \overline{\lambda_f(n)} n^{it} W \left(\frac{n}{N} \right),
 \end{align}
where U is a smooth bump function supported on the interval $[0.5,3]$ such that $U(x)\equiv 1$  for $x\in [1,2]$ and $U^{(j)}(x)\ll_j 1,$ for all $j \geq 1$. Estimating $\mathcal{F}$ trivially at this stage, we get $\mathcal{F}\ll QKN^2$.

{\bf Step  1:} 
On applying the Petersson Trace formula, we obtain 
$\mathcal{F}=\Delta + \mathcal{O}$,
where 
\begin{align*}
\Delta =\sum_{k\sim K} W\left(\frac{k-1}{K}\right) \sideset{}{^\dagger}\sum_{ \psi (\textrm{mod}\ q)}\mathop{\sum \sum}_{m, \ell=1}^\infty \lambda_F(m)\psi(\ell)  U \left(\frac{m \ell^2}{N} \right) m^{it} W \left(\frac{m}{N} \right),
 \end{align*} and 
\begin{align}\label{off diagonal}
 \mathcal{O} &=\sum_{k\sim K} W\left(\frac{k-1}{K}\right) \sideset{}{^\dagger}\sum_{ \psi (\textrm{mod}\ q)}\mathop{\sum \sum \sum}_{m, \ell ,n=1}^\infty \lambda_F(m) n^{it} \psi(\ell) W\left(\frac{n}{N}\right) U \left(\frac{m \ell^2}{N} \right)\notag \\
 & \hspace{2cm}  \times 2\pi i^{-k} \sum_{c=1}^\infty \frac{S_{\psi} (m,n,cq)}{cq} J_{k-1} \left(\frac{4\pi\sqrt{mn}}{cq}\right).
 \end{align}
 We observe that $|\Delta| \asymp KQ|S(N)|$.

 {\bf Step 2:}
  Next we evaluate the sum over $k$ in \eqref{off diagonal} and observe that $\mathcal{O}$ is negligibly small if $QK^{2}\gg Nt^{\epsilon}$. Hence we obtain  $S(N) \ll \frac{\mathcal{F}}{QK}$. Now our goal is to prove that $\mathcal{F}\ll QKN^{1/2}t^{1/3}$.

{\bf Step 3:}
Now we apply functional equation for $L(s, F\otimes f)$ in \eqref{Fourier sum}. We observe that the sum over $m$ in \eqref{Fourier sum} is given by 
\begin{align*}
\mathop{\sum \sum}_{m, \ell=1}^\infty \lambda_f(m)\lambda_F(m)\psi(\ell) U\left(\frac{m \ell^{2}}{N}\right)  &=  
\eta i^{-2k} \left( \frac{N}{\tilde{N}}\right)^{1/2} \epsilon_{\psi}^{2}\overline{\lambda_f(q^{2}) } \sum_{\mathcal{U}}\mathop{\sum \sum}_{m, \ell=1}^\infty \lambda_f(m)\lambda_F(m)  \\ & \hspace{1cm} \times W_1\left(\frac{m \ell^{2}}{\tilde{N}}\right)+ O(t^{-2018}),
\end{align*}
 where  $W_{1}^{(j)}(x)\ll_{j} t^{\epsilon j}$ and $\tilde{N} \asymp Q^2 K^4/N$.
This gives us the following expression of $\mathcal{F}$:

\begin{align} \label{newsum}
\mathcal{F} &=\eta  \left( \frac{N}{\tilde{N}}\right)^{1/2}  \sum_{k\sim K} i^{-2k} W\left(\frac{k-1}{K}\right)
 \sideset{}{^\dagger}\sum_{ \psi (\textrm{mod}\ q)}\epsilon_{\psi}^{2}
 \sum_{f\in H_k(q,\Psi)}\omega_f^{-1}\sum_{\nu=0}^{\infty}\mathop{\sum \sum}_{m^{\prime},  \ell=1}^\infty \lambda_f(m^{\prime})\overline{\lambda_F} (m) \overline{\psi}(\ell m^{\prime})    \notag\\
 &  \hspace{3cm}\times W_{1} \left(\frac{m^{\prime} q^{\nu} \ell^2}{\tilde{N}} \right) \sum_{n=1}^\infty \overline{\lambda_f(nq^{2+\nu})} n^{it} V \left(\frac{n}{N} \right).
\end{align}
This step gives us a saving of the size $(\frac{N}{\tilde{N}})^{1/2} = \frac{N}{QK^2}.$

{\bf Step 4:}
We again apply the Petersson Trace formula in \eqref{newsum} and obtain 
$\mathcal{F}=\textrm{diagonal}(\Delta_{1})+ \textrm{off diagonal}(\mathcal{O^{*}})$.
We observe that  diagonal($\Delta_{1}$) term vanishes and the dual off diagonal term is given by
\begin{align}\label{new off}
\mathcal{O^{*}} &= \left( \frac{N}{\tilde{N}}\right)^{1/2}     \sum_{k\sim K} i^{-2k} W\left(\frac{k-1}{K}\right) \sideset{}{^\dagger}\sum_{ \psi (\textrm{mod}\ q)}\epsilon_{\psi}^{2}\sum_{\nu=0}^{\infty} \mathop{\sum \sum}_{m^{'}, \ell=1}^\infty \lambda_F(m)\overline{\psi}(\ell m^{'})  W_{1} \left(\frac{m^{'} q^{\nu} \ell^2}{\tilde{N}} \right)  \notag \\
 &  \hspace{1cm}\times \sum_{n=1}^\infty  n^{it} V \left(\frac{n}{N} \right)2\pi i^{-k} \sum_{c=1}^\infty \frac{S_{\psi} (nq^{2+\nu},m^{'},cq)}{cq} J_{k-1} \left(\frac{4\pi\sqrt{m^{'}q^{\nu}n}}{c}\right) .
\end{align}

This step gives us a saving of the size $\frac{\sqrt{QK}}{C}$, where $C\backsim Q$. From now on, we shall estimate $\mathcal{O^{*}}$.

{\bf Step 5:}
 We  evaluate the sum over $k$ in \eqref{new off} using stationary phase integral and also evaluate the sum over $\psi$ in the resulting expression. This process gives us the following expression for $\mathcal{O^{*}}$:
\begin{align*}
\mathcal{O^{*}} &=\sqrt{\frac{N}{\tilde{N}}} \frac{\phi(q)}{q}\mathop{\sum \sum}_{m, \ell=1}^\infty \lambda_F(m)  U \left(\frac{m \ell^2}{\tilde{N}} \right) \sum_{n}n^{it} W \left(\frac{n}{N} \right) \sum_{c\ll \frac{Q}{\ell}} \frac{S(n,m;c)}{c} e\left(\pm \frac{\overline{cl}}{q}\pm\frac{\sqrt{nm}}{c} \right) .  
\end{align*} 
In this step, the sum over $k$ gives us a saving of size $\sqrt{K}$ and the sum over $\psi$ gives a saving of  size $\sqrt{Q}$. Thus, saving in this step is 
\begin{align*}
\frac{N}{QK^2}\frac{\sqrt{QK}}{\sqrt{C}}\sqrt{QK}=\frac{N}{K\sqrt{Q}}.
\end{align*}
Hence total saving at this stage is 
\begin{align*}
\frac{N}{\sqrt{Qt}}\frac{N}{K\sqrt{Q}}=\frac{N^2}{QK\sqrt{t}}.
\end{align*}

{\bf Step 6:} We now apply the Poisson summation formula to the sum over $n$. The initial length for $n$-sum is $N$. Here ``analytic conductor'' is of size $t$ and ``arithmetic conductor'' is of size $c$. Hence the dual length is supported on $\frac{ct}{N}$. Saving in this step is of size  $\frac{N}{\sqrt{Qt}}$.  We obtain the following bound for $\mathcal{O^{*}}$:   
\begin{align*}
\mathcal{O^{*}} &\ll \left( \frac{N}{\tilde{N}}\right)^{1/2}   \mathop{\sum \sum}_{m, \ell =1}^\infty |\lambda_F(m)\overline{\psi}(l ) W\left(\frac{ml^{2}}{\Tilde{N}}\right)|\ \ |\sum_{c\sim C} \sum_{n\ll\frac{ct}{N}} \frac{1}{c} e\left(\frac{-m \overline{n}}{c} \right)I(m,n,c)|.
\end{align*}

{\bf Step 7:}
We apply the Cauchy-Schwartz inequality to get rid of the Fourier coefficients. Opening the absolute value square and interchanging the sum over $m$  gives us the following expression:
\begin{align*} 
\mathcal{O^{*}} & \ll \left( \frac{N}{\tilde{N}}\right)^{1/2}     (\tilde{N})^{1/2}N  \left(\sum_{m\sim \tilde{N}} W(\frac{m}{\tilde{N}})\ \  |\sum_{c\sim C}  \sum_{n\ll\frac{c t}{N}} \frac{1}{c} e\left(\frac{-m \overline{n}}{c} \right)I(m,n,c)|^{2}  \right)^{1/2} \notag\\ 
&:=   N^{3/2}  (\mathcal{O}_2^\star)^\frac{1}{2},
\end{align*}
where \begin{align*}
\mathcal{O}_2^\star =\mathop{\sum \sum}_{c_{1},c_{2} \sim Q} \frac{1}{c_1 c_2} \mathop{\sum \sum}_{n_i \sim \frac{c_i t}{N}}  \sum_{m\sim \tilde{N}} e\left(\frac{-m \overline{n_{1}}}{c_{1}} +\frac{m \overline{n_{2}}}{c_{2}} \right)I(m,n_{1},c_{1})I(m,n_{2},c_{2}) U\left(\frac{m}{\tilde{N}}\right).
\end{align*}
We again apply the Poisson summation formula to sum over $m$.  ``Analytic conductor'' is of size $K^2$ and ``arithmetic conductor'' is of size $c_{1}c_{2}$. From the diagonal terms  we get a saving of size $\frac{Q^2 t}{N}$. From off diagonal terms we save $\frac{\tilde{N}}{K^2 \sqrt{c_{1}c_{2}}}$. Also we are able to save $\sqrt{c_{1}c_{2}}$  from the resulting congruence relation. Thus,  total savings in the off diagonal terms is of size  
\begin{align*}
\frac{\tilde{N}}{K^2 \sqrt{c_{1}c_{2}}}\sqrt{c_{1}c_{2}}=\frac{\tilde{N}}{K}=\frac{Q^2 K^3}{N}.
\end{align*}
Therefore,  total savings in sum over $m$ is of size $\min\left\lbrace\frac{Q^2 t}{N},\frac{Q^2 K^3}{N}\right\rbrace$.
Optimal choice of $K$ is given by $K=t^{1/3}$.  Hence, total saving from all of the above steps is of size 
\begin{align*}
\frac{N^2}{QK\sqrt{t}}\left(\frac{Q^2t}{N}\right)^{1/2}=\frac{N^{3/2}}{K}.
\end{align*}
Finally, we obtain 
\begin{align*}
|\mathcal{F}|\ll \frac{QK}{N^{3/2}/K}=KQ\sqrt{N}t^{1/3}.
\end{align*}
This proves our claim.

\section{Preliminaries}
In this section, we shall recall some basic facts about $SL(2, \mathbb{Z})$ automorphic forms (for details see \cite{HI} and \cite{IK1}). 
\subsection{Holomorphic cusp forms} 

Let $f $  be a holomorphic Hecke eigenform of weight $k$ for the full modular group $ SL(2, \mathbb{Z})$.  The   Fourier expansion of $f$  at $\infty$ is given by
$$ f(z)= \sum_{n=1}^\infty \lambda_f(n) n^{(k-1)/2} e(nz),$$
where $ e(z) = e^{2\pi i z}$ and $\lambda_f(n), \ {n \in \mathbb{Z}}$ are the normalized Fourier coefficients.   Deligne proved that $|\lambda_f(n)| \leq d(n)$, where $d(n)$ is the divisor function. $L$-function associated with the form $f$ is given by 

\[
L( s, f )= \sum_{n=1}^\infty \frac{\lambda_f(n)}{n^s} \ =  \prod_p \left( 1 -\lambda_f(p) p^{-s} + p^{-2s} \right)^{-1} \ \ \ (\Re s>1). 
\]  The completed $L$-function is given by 
\[
\Lambda(s, f) : = ( 2 \pi)^{-s} \Gamma \left( s + \frac{k-1}{2}\right) L( s, f ) = \pi^{-s} \Gamma\left( \frac{s + (k+1)/2}{2}\right)  \Gamma\left( \frac{ s + (k-1)/2}{2}\right)L( s, f ). 
\]

 Hecke proved that $L(s, f)$ admits an analytic continuation to the whole complex plane and satisfies the functional equation
\begin{align*} 
 \Lambda(s, f) = \epsilon(f) \  \Lambda(1-s,\overline{f}),
\end{align*}
 where $ \epsilon(f)$ is  a root number and $\overline{f} $ is the dual form of $f$.

\subsection{Maass cusp forms} Let $f$ be a weight zero Hecke-Maass cusp form with Laplace eigenvalue $1/4 + \nu^2$. The Fourier series expansion of $f$ at  $\infty$ is given by 
\[
f(z)= \sqrt{y} \sum_{n \neq 0} \lambda_f(n) K_{ i \nu} (2 \pi |n|y) e(nx), 
\] 
where $ K_{ i \nu}(y)$ is the  Bessel function of  second kind. Ramanujan-Petersson conjecture predicts that $|\lambda_f(n)|\ll n^\epsilon$.  The work of H. Kim and P. Sarnak \cite{KS} tells us  that  $|\lambda_f(n)|\ll n^{7/64+\epsilon}$. $L$-function associated with the form $f$ is defined by $ L(s, f) := \sum_{n=1}^\infty  \lambda_f(n) n^{-s}$ ( $\Re \ s>1$). It extends to an entire function and satisfies the functional equation 
$ \Lambda(s, f) = \epsilon(f ) \Lambda(1-s, \overline{f})$, where $ |\epsilon(f )| = 1$   and completed $L$-function  $ \Lambda(s, f)$ is given by
\[
\Lambda(s, f) = \pi^{-s} \Gamma \left( \frac{s  + i \nu   }{ 2}  \right)   \Gamma \left( \frac{s  - i \nu }{ 2} \right) L(s, f) . 
\]

\section{Some Lemmas}

In this section we shall recall some results which we  require in the sequel. We first recall the following version of the Stirling’s formula. 

\begin{lemma} \label{stirling}
Let $s = \sigma + it$ with $A_1 \leq A_2$ and $t \geq 0.$We have
\begin{equation}
\Gamma(s) = \sqrt{\frac{2\pi}{s}} \left(\frac{s}{e}\right)^{s} \left\lbrace \sum_{1}^{N} \frac{a_{n}}{s^{n}}+O\left( |s|^{-N-1}\right)\right\rbrace,   
\end{equation} and 
\begin{equation*}
|\Gamma(s)| = \sqrt{2\pi} t^{\sigma-1/2} e^{-\frac{\pi}{2} |t|} \left(  1+ O\left( |t|^{-1}\right)\right). 
\end{equation*}
\end{lemma}

\begin{lemma}\label{sum over k}
Let $g(u)$ be a real valued smooth function of $\mathbb{R}$. Let $  \hat{g}(v)$  be the Fourier transform of $g$ and let $J_{u}(x)$ be the Bessel's functions of order $u$. We have
\begin{align*}
4 \sum_{u \equiv a (4)} g(u) J_u (2 \pi x) = \int_{\mathbb{R}} \hat{g}(v) C_a(v, x) dv, 
\end{align*} where 
\begin{align*}
 C_a(v, x) = -2i \sin(x \sin 2 \pi v) + 2 i^{1-a} \sin(x \cos 2 \pi v). 
\end{align*}
\end{lemma}

\begin{proof}
See \cite[page 85-86]{HI}.
\end{proof}

We now recall  Rankin-Selberg bound for Fourier coefficients in the following lemma. 

\begin{lemma} \label{rankin Selberg bound}
Let $\lambda_f(n)$ be Fourier coefficients of a holomorphic cusp form, or a Maass form. For any real number $x\geq 1$, we have 
\begin{align*}
\sum_{1\leq n \leq x} \left| \lambda_f(n) \right|^2 \ll_{f, \epsilon} x^{1+\epsilon}. 
\end{align*} 

\end{lemma}

We also require to estimate the exponential integral of the form: 
\begin{equation} \label{eintegral}
\mathfrak{I}= \int_a^b g(x) e(f(x)) dx,
\end{equation} where $f$ and $g$ are  real valued smooth functions on the interval $[a, b]$. We recall the following lemma on exponential integrals.

\begin{lemma} \label{second deri bound}
Let $f$ and $g$ be real valued twice differentiable function and let $f^{\prime \prime} \geq r>0$ or  $f^{\prime \prime} \leq -r <0$, throughout the interval $[a, b]$. Let $g(x)/f^\prime(x)$ is monotonic and $|g(x)| \leq M$. Then we have

\begin{align*}
\mathfrak{I} \leq \frac{8M}{\sqrt{r}}. 
\end{align*} 
\end{lemma} 
\begin{proof}
See \cite[Lemma 4.5, page 72]{ECT}
\end{proof}

\begin{lemma} \label{exponential inte}
Let $0<\delta<1/10$, $X,\ Y,\ V,\ V_{1},\ Q \ >0$, $Z:=Q+X+Y+V_{1}+1$, and assume that 

\begin{align} 
Y\ge Z^{3\delta},\ \ V_{1} \ge V \ge  \frac{QZ^{\frac{\delta}{2}}}{Y^{\frac{1}2{}}}. \ \ \ 
\end{align} Suppose that $w$ is a smooth function on $\mathbb{R}$ with support on an interval $J$ of length $V_{1}$ satisfying $w^{(j)}\ll_{j} XV^{-j}$, for all $j\in \mathbb{N}$. Suppose that $h$ is a smooth function on $J$ such that there exists a unique point $t_{0}\in J$ such that $h^\prime(t_0)=0$, and furthermore that
\begin{align} 
h^{(2)}(t) \gg YQ^{-2}, \ \ h^{(j)}(t) \ll_{j} YQ^{-j},\ \  for j=1, 2,...\ and\ \ t\in J 
\end{align}
Then the integral $I$ defined by 
\begin{align*}
    I=\int_{\mathbb{R}} w(t)e^{i h(t)} dt 
\end{align*}
has an asymptotic expansion of the form 

\begin{align} \label{huxely bound}
I=\frac{e^{ih(t_{0})}}{\sqrt{h^{(2)}(t_{0})}}\sum_{n\le 3\delta^{-1}A} p_{n}(t_{0})  + O_{A,\delta}(Z^{-A}),
\end{align}
\begin{align*}
   p_{n}(t_{0})=\frac{\sqrt{2\pi}e^{\pi i/4}}{n!}\left(\frac{i}{2h^{(2)}(t_{0})}\right)^{n} G^{(2n)}(t_{0}), 
\end{align*}
 where $A>0$ is arbitrary,\ and 
\begin{align}
    G(t)=\ w(t)e^{iH(t)} ,\ \ H(t)=\ h(t)-\ h(t)-\ \frac{1}{2}h^{(2)}(t_{0})(t-\ t_{0})^{2}.
\end{align}
Furthermore,\ each $p_{n}$ is a rational function in $h^{\prime \prime},\ h^{\prime \prime \prime},...,$ satisfying 
\begin{align}
    \frac{d^{j}}{dt^{j}} p_{n}(t_{0})\ll_{j,n}X(V^{-j}+\ Q^{j})\left( (V^2 Y/Q)^{-n} +\ Y^{-n/3}\right).
\end{align}

The leading term satisfies
\begin{align*}
    \sqrt{2\pi}e^{\frac{\pi i}{4}}\frac{e^{i h(t_{0})}}{\sqrt{h^{(2)}(t_{0})}}w(t_{0}) \ll \frac{Q  X}{Y^{1/2}}.
\end{align*} 
Also, if $h(t)$ does not vanishes on the interval $J$ and satisfies $|h^\prime(t)| \geq R$ for some $R>0$, then  we have
\begin{align} \label{without stationary}
    I \ll_A V X \left[ (QR/ \sqrt{Y})^{-A} + (RV)^{-A} \right].
\end{align}

\end{lemma}
\begin{proof}
See Lemma $8.1$ and and Proposition $8.2$  of \cite{BKY}.  We use this result to show that in absence of stationary phase, the integral is negligibly small, i.e., $O_A(t^{-A})$ for any $A>0$, if $R\gg t^\epsilon \max \left\lbrace Y^{1/2}/ Q, V^{-1}\right\rbrace$.
\end{proof}
\section{First application of the Petersson trace formula}
To prove our theorem, we shall prove the following proposition. 
\begin{proposition}
\begin{align*}
S(N) \ll 
\begin{cases}
N \ \ \ \ \ \ \ \ \ \ \ \ \ \ \ \textrm{if} \ \   1 \ll  N\ll t^{2/3 + \epsilon} \\
\sqrt{N} t^{1/3 + \epsilon}  \ \ \ \ \ \textrm{if} \ \    t^{2/3 + \epsilon} \ll N \ll t^{1+\epsilon}    
\end{cases},
\end{align*}
\end{proposition} where
\begin{align*}
S(N)=\sum \lambda_F(n)n^{it}W(n/N).
\end{align*}
We shall use the Petersson trace formula to separate the oscillations of $\lambda_F(n)$ and $n^{it}$,
where we use harmonics from $H_k(q,\Psi)$, with $k\backsim K$ , $q\backsim Q$ and $\Psi$ is an odd character. Optimal size of K and Q will be choosen later. We now consider the following Fourier sum:

 \begin{align}\label{sumit}
 \mathcal{F} &=\sum_{k\sim K} W\left(\frac{k-1}{K}\right) \sideset{}{^\dagger}\sum_{ \psi (\textrm{mod}\ q)}\sum_{f\in H_k(q,\Psi)}\omega_f^{-1}\mathop{\sum \sum}_{m, \ell=1}^\infty \lambda_f(m)\lambda_F(m)\psi(\ell)  U \left(\frac{m \ell^2}{N} \right) \nonumber \\
 & \hspace{3cm}\times \sum_{n=1}^\infty \overline{\lambda_f(n)} n^{it} W \left(\frac{n}{N} \right),
 \end{align}
where U is a smooth bump function supported on the interval $[0.5,3]$ such that $U(x)\equiv 1$  for $x\in [1,2]$ and $U^{(j)}(x)\ll_j 1,$ for all $j \geq 1$. On applying the Petersson trace formula to the above sum $\mathcal{F}$, we observe that the diagonal term is given by
 \begin{align*}
\Delta =\sum_{k\sim K} W\left(\frac{k-1}{K}\right) \sideset{}{^\dagger}\sum_{ \psi (\textrm{mod}\ q)}\mathop{\sum \sum}_{m, \ell=1}^\infty \lambda_F(m)\psi(\ell)  U \left(\frac{m \ell^2}{N} \right) m^{it} W \left(\frac{m}{N} \right).
 \end{align*}
 Since $W$ is supported on $[1,2]$, the above sum is non-zero only when $N\leq m\leq 2N $. If $\ell \geq 2$, then $m\ell^{2} \geq 4N $.  This gives us $\frac{m\ell^{2}}{N} \geq 4$. Since $U(x)$ vanishes for $x\geq 3$, this forces $\ell =1$. Finally we obtain the following expression for $\Delta$:
 \begin{align*}
\Delta =\sum_{k\sim K} W\left(\frac{k-1}{K}\right) \sideset{}{^\dagger}\sum_{ \psi (\textrm{mod}\ q)}\sum _{m=1}^\infty \lambda_F(m) m^{it} W \left(\frac{m}{N} \right)
 \end{align*}
 \begin{equation*}
 \Rightarrow|\Delta| \asymp KQ|S(N)|.
 \end{equation*}

Next we consider the off diagonal term, which is given by
  \begin{align}\label{ooff}
\mathcal{O}  &=\sum_{k\sim K} W\left(\frac{k-1}{K}\right) \sideset{}{^\dagger}\sum_{ \psi (\textrm{mod}\ q)}\mathop{\sum \sum \sum}_{m, \ell ,n=1}^\infty \lambda_F(m) n^{it} \psi(\ell) W\left(\frac{n}{N}\right) U \left(\frac{m \ell^2}{N} \right) \notag\\
&  \hspace{3cm} \times 2\pi i^{-k} \sum_{c=1}^\infty \frac{S_{\psi} (m,n,cq)}{cq} J_{k-1} \left(\frac{4\pi\sqrt{mn}}{cq}\right).
 \end{align}
 We will now consider the sum over $k$ in  the above equation. Using Lemma \ref{sum over k} with $x=\frac{2\sqrt{mn}}{cq}$, we obtain
 \begin{align*}
S_{1} :  = \sum_{k \sim K} i^{-k}W\left(\frac{k-1}{K}\right)J_{k-1}(2\pi x)
= \int_{\mathbb{R}} \hat{W}\left(\frac{v}{K}\right) \sin(x  \cos(2\pi v))dv \\
   = K \int_{\mathbb{R}} \hat{W}(Kv)) \sin(x  \cos(2\pi v))dv.
\end{align*}
By change of  variable $Kv\rightarrow v$ in above equation,  we obtain
\begin{align*}
S_{1}=  \int_{\mathbb{R}} \hat{W}(v)) e(x\cos(\frac{2\pi v}{K}))dv .
\end{align*}

We have
 \begin{equation} \label{w hat}
 \hat{W}(v)=\int_{\mathbb{R}}W(u)e(vu)du.
\end{equation}
 Integrating by parts $j$ times and using $W^{(j)}(u)\ll_{j}(t^{\epsilon})^{j}$,  we obtain 
\begin{align*}
\hat{W}(v)\ll\left(\frac{t^{\epsilon}}{2\pi v}\right)^{j}.
\end{align*}
Hence, $\hat{W}(v)$ is negligibly small if $|v|\gg t^{\epsilon}$. Taking  $V=At^{\epsilon}$ for some fixed constant $A$, we get 
\begin{align} \label{S_1}
S_{1}=  \int_{\mathbb{R}} \hat{W}(v)) F(v) e(x \cos(\frac{2\pi v}{K}))dv.
\end{align}
Where $F$ is a smooth bump function supported on the interval $[-2V,2V]$ such that $F(v)\equiv 1$  for $v\in [-V,V]$ and $F^{(j)}(v)\ll_j 1,$ for all $j \geq 1$.
Using equation \eqref{w hat} in equation \eqref{S_1},  we obtain
\begin{align*}
S_{1}=  \mathop{\int\int}_{\mathbb{R}^{2}} \hat{W}(v)) F(v) e\left( uv \pm x \cos(\frac{2\pi v}{K}) \right)dudv.
\end{align*}

Applying Lemma \ref{exponential inte}  to the $v$-integral, we observe  that the integral is negligibly small if $x\ll K^{2-\epsilon}$. This analysis holds even if the weight function has the little oscillation, say $W^{j}\ll_{j} t^{j\epsilon}$. In the complementary range for $x$ we expand the cosine function into a Taylor series. Since $x\ll N/Q$, if we assume that $N\ll QK^{4}t^{-\epsilon}$, then we only need to retain the first two terms in the expansion, and the above integral essentially reduces to  
\begin{align*}
    e(\pm x)\int\int_{\mathbb{R}^{2}} W(u))F(v) e\left(uv\pm \frac{4\pi^{2}xv^{2}}{K^{2}}\right)dudv.
\end{align*}
For integral over $v$, we apply the stationary phase analysis. If we choose $+$ sign in the above equation, then $v$ integral is negligibly small due to absence of stationary point (by second case of Lemma \ref{exponential inte}). Otherwise,  the integral is given by
\begin{align*}
    e\left(x+\frac{u^{2}K^{2}}{16\pi^{2}x}\right)\frac{K}{\sqrt{x}}\rightsquigarrow e(x)\frac{K}{\sqrt{x}}
\end{align*}
with $x\gg K^{2-\epsilon}$ (upto an oscillatory factor which oscillates at most like $t^{\epsilon}$). In any case, it follows that  we can cut the sum over $c$ in \eqref{ooff} at $C\gg Nt^{\epsilon}/QK^{2}$, at a cost of a negligible error term. Hence effective range of $x$ is given by $x\gg K^{2}t^{\epsilon}$.
We note that $x=\frac{2\sqrt{mn}}{cq}$.  Hence we have $ \frac{N}{cQ}\gg K^{2} t^{\epsilon}$ i.e., $ c\ll\frac{N t^{\epsilon}}{QK^{2}}$. If we choose parameters $K$ and $Q$ such that $QK^{2}\gg Nt^{\epsilon}$, the off diagonal term is negligibly small. Hence, We obtain 
\begin{align*}
|S(N)|\ll \frac{|\mathcal{F}|}{QK} +t^{-2018}.
\end{align*}
\section{functional equation for $L(s,F\otimes f)$ }
 We now  consider the sum 
\begin{align*}
S_2 := \mathop{\sum \sum}_{m, \ell=1}^{\infty} \lambda_f(m) \lambda_F(m)\psi(\ell)  U \left(\frac{m \ell^{2}}{N}\right) .
\end{align*} 
Using Mellin inversion formula,  we obtain 
\begin{align*}
S_2 =\int_{(\sigma)}\tilde{U}(s)N^{s}\mathop{\sum \sum}_{m, \ell=1}^\infty \frac{\lambda_f(m)\lambda_F(m)}{(m\ell ^{2})^{s}}\psi(\ell) ds
=\int_{(\sigma)}\tilde{U}(s)N^{s} L(s,F\otimes f)ds.
\end{align*}
Applying the functional equation for $L(s,F\otimes f)$ (see \cite[page 135-136]{KMV}) in the above equation, we obtain
 \begin{equation*}
 S_2=\frac{q}{2\pi}\frac{g_{\psi}^{2}}{q\lambda_f(q^{2})}\int_{(\sigma)}\tilde{U}(s)\left(\frac{N}{(2\pi q)^{2}}\right)^{s}\frac{\gamma_k(1-s)}{\gamma_k(s)}  L(1-s,\overline{F}\otimes \overline{f})ds ,
 \end{equation*}
where $\gamma_{k}(s)$ is a  product of four gamma factors. Moving the line of integration to $\sigma=-\epsilon$ and expanding the resulting $L$ function into series, we obtain
\begin{align*}
S_2= \frac{g_{\psi}^{2}}{2 \pi i} \overline{\lambda_f(q^{2}) }  \mathop{\sum \sum}_{m, \ell=1}^\infty \frac{\lambda_f(m)\lambda_F(m)}{m\ell ^{2}}\psi(\ell) U\left(\frac{m \ell^{2}}{\tilde{N}}\right) \int_{(-\epsilon)}\tilde{U}(s)\left(\frac{Nm\ell^{2}}{(q)^{2}}\right)^{s}\frac{\gamma_k(1-s)}{\gamma_k(s)})ds.
\end{align*}
For $\tilde{N}\gg \frac{Q^{2}K^{4 }}{N}t^{\epsilon}$, we shift the contour to the left  and for $\tilde{N}\ll \frac{Q^{2}K^{4 }}{N}t^{-\epsilon}$,  we shift the contour to $\frac{K-2}{2}$. Since $K$ is  of size $\gg t^{1/3 - \epsilon}$, we observe that the contribution from the above range is negligibly small. Let $\mathcal{U}={(U,\tilde{N})}$ be a smooth  dyadic  partition of unity, which consists of pair $(U,\tilde{N})$ with $U$ a non negative smooth function on $[1,2]$ and $\sum_{(U,\tilde{N})} U\left(\frac{r}{\tilde{N}}\right)= 1$ for  $r \in (0,\infty)$. Also the collection is such that the sum is locally finite in the sense that for any given $\ell \in \mathbb{Z}$, there are only finitely many pairs with $\tilde{N} \in [2^{\ell},2^{\ell +1}]$. We record the above result in the following lemma.
\begin{lemma} Let $\tilde{N}$ be as above. We have
\begin{align*}
&\mathop{\sum \sum}_{m, \ell=1}^\infty \lambda_f(m)\lambda_F(m)\psi(\ell) U\left(\frac{m \ell^{2}}{N}\right)= \eta i^{-2k} q \ \epsilon_{\psi}^{2}\overline{\lambda_f(q^{2}) } \sum_{\mathcal{U}}\mathop{\sum \sum}_{m, \ell=1}^\infty \frac{\lambda_f(m)\lambda_F(m)}{m\ell ^{2}}\psi(\ell)   \\
 & \hspace{3cm}\times U\left(\frac{m \ell^{2}}{\tilde{N}}\right)\frac{1}{2 \pi i} \int_{(0)}\tilde{U}(s)\left(\frac{Nm\ell^{2}}{(q)^{2}}\right)^{s}\frac{\gamma_k(1-s)}{\gamma_k(s)})ds +O(t^{-2018}),
\end{align*}
  where $\tilde{N}\asymp \frac{Q^{2}K^{4}}{N}$.
\end{lemma}  To cancel out the oscillations of Gamma functions, we  shift the contour to $\Re s = 1/2$. For $s=\sigma +i\tau$, we have 
\begin{align*}
\tilde{U}(s)=\int_{0}^{\infty} U(x)x^{s-1}dx \ll_{j} \frac{t^{\epsilon}}{|s||s+1| \cdots |s+j-1|}.
\end{align*}
We observe that $\tilde{U}(s)$ is negligibly small if $\tau \gg t^{\epsilon}$. Hence we shall focus on the range $|\tau|\ll t^{\epsilon}$. We note that $\frac{\gamma_k(1/2 +i\tau)}{\gamma_k(1/2-i\tau)}$ is a  product of $4$ factors of the form $\frac{\Gamma(k_{j} +i\tau)}{\Gamma(k_{j}-i\tau)}$, where $k_{j}\backsim K $. Let $K/2 +i\tau= re^{i\theta}$ with $r=\sqrt{\frac{K^{2}}{4}+\tau ^{2}}$ and $\theta=\tan^{-1}(2\tau /K)$.  We obtain $\log r= \log K+O(\tau^{2}/K^2)$ and $\theta=\tau /K +O(\tau^3/K^3)$.
Using lemma \ref{stirling}, we obtain
\begin{align*}
\frac{\Gamma(K/2 +i\tau)}{\Gamma(K/2-i\tau)} &=\exp\left\lbrace (K/2+i\tau -1/2)\log(re^{i\theta}) -(K/2+i\tau)-(K/2-i\tau -1/2)\log(re^{-i\theta})  \right. \\ 
  & \hspace{7cm} \left. +K/2+i\tau +O(\tau/K)\right\rbrace \\
&=\exp\left\lbrace (K/2+i\tau -1/2)(\log K+O(\tau^2 /K^2) +i(\tau/K +O(\tau^3/K^3))) \right. \\
& \left. \hspace{10pt} -(K/2+i\tau)-(K/2-i\tau -1/2)( \log K+O(\tau^2 /K^2) +i(\tau/K +O(\tau^3/K^3)))    \right\rbrace \\
&=\exp(2i\tau \log K -i\tau /K -i \tau +O(\tau^2/K^2).
\end{align*} 
 
 We observe that oscillations with respect to $k$ are given by $(k/2)^{i\tau}$. Since $\tau \ll t^{\epsilon}$, we can ignore the oscillations with respect to $k$ and replace $W$ by $W_{1}$ such that $W_{1} ^{(j)}(x)\ll t^{j \epsilon}$. We record the above result in the following lemma. 
\begin{lemma} \label{lemma F times f}
Let $F$ and $f$ be as above. We have
\begin{align*} 
&S_3:= \mathop{\sum \sum}_{m, \ell=1}^\infty \lambda_f(m)\lambda_F(m)\psi(\ell) U\left(\frac{m \ell^{2}}{N}\right) = \eta i^{-2k} N^{1/2} \epsilon_{\psi}^{2}\overline{\lambda_f(q^{2}) } \sum_{\mathcal{U}}\mathop{\sum \sum}_{m, \ell=1}^\infty \frac{\lambda_f(m)\lambda_F(m)}{\sqrt{m\ell ^{2}}}\psi(\ell)   \\
 & \hspace{3cm}\times U\left(\frac{m \ell^{2}}{\tilde{N}}\right)\frac{1}{2 \pi } \int_{(1/2)}\tilde{U}(s)\left(\frac{Nm\ell^{2}}{(q)^{2}}\right)^{it}\frac{\gamma_k(1-s)}{\gamma_k(s)})dt +O(t^{-2018}) \\
 &\hspace{2cm} = \eta i^{-2k} \left( \frac{N}{\tilde{N}}\right)^{1/2} \epsilon_{\psi}^{2}\overline{\lambda_f(q^{2}) } \sum_{\mathcal{U}}\mathop{\sum \sum}_{m, \ell=1}^\infty \lambda_f(m)\lambda_F(m)    W_1\left(\frac{m \ell^{2}}{\tilde{N}}\right) +O(t^{-2018}),
\end{align*}
  where  $W_{1}^{(j)}(x)\ll_{j} t^{\epsilon j}$.
\end{lemma}

 Let $m=m^{\prime}q^{\nu}$,  where $(m^{\prime},q)=1$.
 Then \begin{equation*}
 \overline{\lambda_{f}(mq^{2})}=\overline{\lambda_{f}(m^{\prime})} \  \overline{\lambda_{f}(q^{2+\nu})}=\overline{\psi(m^{\prime})}\lambda_{f}(m^{\prime}) \ \overline{\lambda_{f}(q^{2+\nu})}.
 \end{equation*}
Using the above expression and Lemma \ref{lemma F times f}  in \eqref{sumit}, we obtain      
\begin{align*}
\mathcal{F} &=\eta  \left( \frac{N}{\tilde{N}}\right)^{1/2}  \sum_{k\sim K} i^{-2k} W\left(\frac{k-1}{K}\right) \sideset{}{^\dagger}\sum_{ \psi (\textrm{mod}\ q)}\epsilon_{\psi}^{2} \sum_{f\in H_k(q,\Psi)}\omega_f^{-1}\sum_{\nu=0}^{\infty}\mathop{\sum \sum}_{,m^{\prime}, \ell=1}^\infty \lambda_f(m^{\prime})\overline{\lambda_F(m)}\overline{\psi}(\ell m^{\prime})    \\
 &  \hspace{3cm}\times W_{1} \left(\frac{m^{\prime} q^{\nu} \ell^2}{\tilde{N}} \right) \sum_{n=1}^\infty \overline{\lambda_f(nq^{2+\nu})} n^{it} V \left(\frac{n}{N} \right).
\end{align*}
Now applying the Petersson trace formula, we obtain

\begin{align*}
\mathcal{F} &=\eta  \left( \frac{N}{\tilde{N}}\right)^{1/2}    \sum_{k\sim K} i^{-2k} W\left(\frac{k-1}{K}\right) \sideset{}{^\dagger}\sum_{ \psi (\textrm{mod}\ q)}\epsilon_{\psi}^{2} \sum_{\nu=0}^{\infty} \mathop{\sum \sum}_{m^{\prime}, \ell=1}^\infty \lambda_F(m)\overline{\psi}(\ell m^{\prime})  W_{1} \left(\frac{m^{\prime} q^{\nu} \ell^2}{\tilde{N}} \right)   \\
 &  \hspace{1cm} \times \sum_{n=1}^\infty  n^{it} V \left(\frac{n}{N} \right) \left\lbrace  \delta(m^{\prime},nq^{2+\nu})+2\pi i^{-k} \sum_{c=1}^\infty \frac{S_{\psi} (nq^{2+\nu},m^{\prime},cq)}{cq} J_{k-1} \left(\frac{4\pi\sqrt{m^{\prime}q^{\nu}n}}{c}\right) \right\rbrace .
\end{align*}
If $m^{\prime}=nq^{2+\nu}$ $\Rightarrow \psi(m^{\prime} \ell)= \psi(nq^{2+\nu}\ell)=0$, as $\psi$ is a character mod $q$. Hence diagonal term vanishes. From now on, we shall consider the dual off diagonal term, which is given by (with a constant multiple of  $\eta$)
\begin{align} \label{o star before proposition}
\mathcal{O^{*}} &= \left( \frac{N}{\tilde{N}}\right)^{1/2}     \sum_{k\sim K} i^{-2k} W\left(\frac{k-1}{K}\right) \sideset{}{^\dagger}\sum_{ \psi (\textrm{mod}\ q)}\epsilon_{\psi}^{2}\sum_{\nu=0}^{\infty} \mathop{\sum \sum}_{m^{'}, \ell=1}^\infty \lambda_F(m)\overline{\psi}(\ell m^{'})  W_{1} \left(\frac{m^{'} q^{\nu} \ell^2}{\tilde{N}} \right)  \notag \\
 &  \hspace{3cm}\times \sum_{n=1}^\infty  n^{it} V \left(\frac{n}{N} \right)2\pi i^{-k} \sum_{c=1}^\infty \frac{S_{\psi} (nq^{2+\nu},m^{'},cq)}{cq} J_{k-1} \left(\frac{4\pi\sqrt{m^{'}q^{\nu}n}}{c}\right) .
\end{align} 
Next, we shall prove the following proposition. 

\begin{proposition} \label{proposition o star}
Let $ \mathcal{O^{*}} $ be as above. We have
\begin{align*}
\mathcal{O^{*}} \ll \sqrt{N} Q K^2 \left(1   + \frac{\sqrt{t}}{K^{3/2}}\right).
\end{align*}
\end{proposition}

To prove our theorem, it is enough to prove the above proposition. Because, when $K= t^{1/3}$, we obtain $ \mathcal{O^{*}} \ll \sqrt{N} Q K t^{1/3}$, which implies that 
\begin{align*}
L\left(\frac{1}{2} + it, F \right) \ll \frac{|S(N)|}{\sqrt{N}} \ll \frac{\mathcal{O^\star}}{\sqrt{N} QK} \ll t^{\frac{1}{3} + \epsilon} .
\end{align*}

\section{Analysis of dual off-diagonal }
We now consider sum over $k$.Using Lemma \ref{sum over k} with $x= 2 \frac{\sqrt{m^\prime q^\nu n}}{c}$, we have
\begin{align*}
	S_{4} & = \sum_{k} i^{-k} W\left(\frac{k-1}{K}\right) J_{k-1}\left(\frac{4\pi \sqrt{m^{\prime} q^{\nu}n}}{c}\right)
= \iint_{\mathbb{R}^{2}} W\left(u \right) F\left(v\right) e\left(uv\pm x \  \cos\frac{2\pi v}{K} \right)du dv. 
\end{align*} The above integral is negligibly small if $c\gg \frac{Q t^{\varepsilon}}{l}$. So effective range of $c$ is given by $c\ll \frac{Q t^{\varepsilon}}{l}$. Since $\cos\left(y\right) =1-\frac{y^{2}}{2} + O\left(y^{4}\right)$ and $u\ll t^{\varepsilon}$,  we obtain
\begin{align*}
S_{4} = e\left(x\right)\int_{\mathbb{R}} W\left(u \right)\int_{\mathbb{R}} F\left(v\right) e\left( uv\pm \frac{x \pi^2 v^{2}}{k^{2}}\right)\left(1+ O\left(\frac{1}{k^{4}}\right)\right) dv  du.
\end{align*}
Let $G\left(v\right)=uv\pm \frac{x \pi^{2} v^{2}}{k^{2}}$. If we choose positive sign in  $G$, then there is no stationary point, so the above integral is negligibly small. From now on, we shall consider $G$ with negative sign. If $G^{\prime}\left(v_{0}\right) =0$, then $v_{0}=\frac{uk^{2}}{4\pi x} \asymp \frac{K^{2}}{x}$; $G^{\prime\prime}\left(v\right)=\pm\frac{4x\pi}{k^{2}}$  and $G^{\left(j\right)}\left(v\right)=0$, for $ j\ge 3$. Applying  Lemma  \ref{exponential inte}, we obtain    

\begin{align*}
S_{4} & = e(x) \int_{\mathbb{R}} W\left(u \right)\frac{F\left(v_{0}\right)}{\sqrt{G^{\prime\prime}\left(v_{0}\right)}}  e\left(G(v_{0})+\frac{1}{8} \right)  du  + \textrm{errors}\\
&= \int_{\mathbb{R}} W\left(u \right)e\left(x + \frac{u^2 K^2}{ 8 \pi^2 x} \right) \frac{K}{\sqrt{x}} (1+ o(1)) du. 
\end{align*} Since $x \gg K^2 t^\epsilon$, we note that the second term in exponential is not oscillating with respect to $x$. We push that term in weight function.  We obtain
  \begin{align*}
  S_{\psi}\left(n q^{2+\nu},m^{\prime};cq\right)=S_{\psi}\left(0,m^{\prime}\overline{c};q\right) S\left(nq^{1+\nu},m^{\prime}\overline{q};c\right) =  \sqrt{q} \ \overline{\epsilon_{\psi}}\psi\left(m^{\prime}\overline{c}\right) S\left(n,m^{\prime}q^{\nu};c\right).
\end{align*}   
We shall now execute the sum over $\psi$, which is given by
 \begin{align*}
 \frac{1}{2}\sum_{\psi(q)}\left(1-\psi(-1)\right)\epsilon_{\psi}^{2}\overline{\epsilon_{\psi}}\psi\left(m^{\prime}\overline{c}\right)\overline{\psi(m^{\prime}l)}
	=\sideset{}{^\pm}\sum_{\psi(q)}\psi\left(\pm  \overline{cl}\right)\frac{1}{2 \sqrt{q}}\sum_{\alpha(q)}\psi(\alpha)e(\frac{\alpha}{q})
= \frac{1}{2 \sqrt{q} }e(\pm\overline{cl})\phi(q).
\end{align*} Substituting the above estimate in equation \eqref{o star before proposition}, we get
\begin{align*}
\mathcal{O^{*}} &=\sqrt{ \frac{N}{\tilde{N}}}  \frac{\phi(q)}{q}\mathop{\sum \sum}_{m, \ell=1}^\infty \lambda_F(m)  U \left(\frac{m \ell^2}{\tilde{N}} \right) \sum_{n}n^{it} W \left(\frac{n}{N} \right) \sum_{c\ll \frac{Q}{\ell}} \frac{S(n,m;c)}{c} e\left(\pm \frac{\overline{cl}}{q}\pm\frac{\sqrt{nm}}{c} \right).  
\end{align*} 

Now we consider the sum over $n$. Let 
\begin{align*}
S_{5}:= \sum_{n} n^{it} e\left(\frac{\sqrt{nm}}{c}\right)S(n,m;c) W \left(\frac{n}{N} \right). 
\end{align*} 
Substituting $n=\alpha+bc$, where  $0\le b<c$, we obtain
\begin{align*}
S_{5}&= \sum_{\alpha(c)} S(\alpha,m;c) \sum_{b} (\alpha+bc)^{it} e\left(\frac{\sqrt{m(\alpha+bc)}}{c}\right) W \left(\frac{\alpha+bc}{N} \right).
\end{align*}

Applying the Poisson summation formula  to the sum over $b$,  we obtain
\begin{align*}
S_{5}&= \sum_{\alpha(c)} S(\alpha,m;c) \sum_{n}\int_{\mathbb{R}} (\alpha+yc)^{it} e\left(\frac{\sqrt{m(\alpha+yc)}}{c}\right) W \left(\frac{\alpha+yc}{N} \right)e\left(-ny\right) dy.
\end{align*}
By the change of variable $v=\frac{\alpha+yc}{N}$,  $dy=\frac{N}{c} dv$, we obtain
\begin{align*}
S_{5}&=\frac{N^{1+it}}{c}\sum_{n}\sum_{\alpha(c)}S(\alpha,m;c)e\left(\frac{n\alpha}{c}\right)\int_{\mathbb{R}} v^{it}W(v)e\left(\frac{\sqrt{mNv}-nNv}{c}\right) dv \\
&=\sum_{n} \mathcal{C}(m,c)I(m,n,c) ,
\end{align*} where $\mathcal{C}(m,c)$ is the character sum and $I(m,n,c)$ is the integral in the above equation. Integrating by parts, we observe that $I(m,n,c)\ll_{j}\left(t+\frac{\sqrt{mN}}{c}\right)^{j} (\frac{c}{nN})^{j}$. We choose the parameter $K$ such that $K\ll t^{1/2-\delta}$. By this choice of K, we obtain $I(m,n,c)\ll_{j}\left( \frac{ct}{nN}\right)^{j}$. We observe that the integral $I(m,n,c)$ is negligibly small if $n\gg \frac{ct^{1+\varepsilon}}{N}$.  Now we consider the character sum $\mathcal{C}(m,c)$, which is given by
\begin{align*}
\mathcal{C}(m,c)=\sum_{\alpha(c)}\sum_{\beta(c)}e\left(\frac{\alpha\beta+m\overline{\beta} +n\alpha}{c} \right)
=\sum_{\beta(c)}e\left(\frac{m\overline{\beta}}{c} \right) \sum_{\alpha(c)} e\left(\frac{\alpha(n+\beta)}{c} \right)
=c\ e\left(\frac{-m \overline{n}}{c} \right).
\end{align*}
Substituting the above estimates for $\mathcal{C}(m,c) $ and $ I(m,n,c)$, we obtain
\begin{align*}
S_{5}=N^{1+it}\sum_{n\ll\frac{Qt}{N}}e\left(\frac{-m \overline{n}}{c} \right)I(m,n,c).
\end{align*}

Now we analyse the integral  $ I(m,n,c)$. We have 
\begin{align*}
I(m,n,c)=\int_{\mathbb{R}}W(v)e\left( \frac{t \log y}{2\pi}+\frac{\sqrt{mNv}-nNv}{c}\right) dv :=\int_{\mathbb{R}}W(v)e(G_{1}(v))dv,
\end{align*}
We note that  $G_{1}^{\prime\prime}(v)=-\frac{t}{2\pi v^{2}} \ -\frac{3\sqrt{mN}}{4cv^{3/2}} \Rightarrow | G_{1}^{\prime\prime}(v)| \asymp t$. By Lemma \ref{second deri bound}, we obtain $I(m,n,c)\ll \frac{1}{\sqrt{t}}$. Substituting the estimate for $S_5$ and using $\phi(q)/q <1$, we  obtain
\begin{align} \label{o star after s5}
\mathcal{O^{*}} &\ll \left( \frac{N}{\tilde{N}}\right)^{1/2}   \mathop{\sum \sum}_{m, \ell =1}^\infty |\lambda_F(m)\overline{\psi}(l ) W\left(\frac{ml^{2}}{\Tilde{N}}\right)|\ \ |\sum_{c\sim C} \sum_{n\ll\frac{ct}{N}} \frac{1}{c} e\left(\frac{-m \overline{n}}{c} \right)I(m,n,c)|.
\end{align} 
  For simplicity, we shall consider the case  $l=1$ (Estimates for the other values of $\ell$ are similar). On applying the Cauchy-Schwartz inequality in equation \eqref{o star after s5}, we obtain 
\begin{align} \label{definition O 2 star}
\mathcal{O^{*}} & \ll \left( \frac{N}{\tilde{N}}\right)^{1/2}     (\tilde{N})^{1/2}N  \left(\sum_{m\sim \tilde{N}} W(\frac{m}{\tilde{N}})\ \  |\sum_{c\sim C}  \sum_{n\ll\frac{c t}{N}} \frac{1}{c} e\left(\frac{-m \overline{n}}{c} \right)I(m,n,c)|^{2}  \right)^{1/2} \notag\\ 
&:=   N^{3/2}  (\mathcal{O}_2^\star)^\frac{1}{2}.
\end{align} Opening the absolute square and interchanging the order of summation, we obtain
\begin{align}\label{O 2 star}
\mathcal{O}_2^\star =\mathop{\sum \sum}_{c_{1},c_{2} \sim Q} \frac{1}{c_1 c_2} \mathop{\sum \sum}_{n_i \sim \frac{c_i t}{N}}  \sum_{m\sim \tilde{N}} e\left(\frac{-m \overline{n_{1}}}{c_{1}} +\frac{m \overline{n_{2}}}{c_{2}} \right)I(m,n_{1},c_{1})I(m,n_{2},c_{2}) U\left(\frac{m}{\tilde{N}}\right).
\end{align}
We now apply the  Poisson summation formula to the sum over $m$ with modulus $c_{1}c_{2}$. Writing $m=\beta +bc_{1}c_{2}$, we get
\begin{align*}
S_{6} &: =\sum_{\beta(c_{1}c_{2})}e\left(\frac{-\beta \overline{n_{1}}}{c_{1}} +\frac{\beta \overline{n_{2}}}{c_{2}} \right)\sum_{b}I(\beta +bc_{1}c_{2},n_{1},c_{1})I(\beta +bc_{1}c_{2},n_{2},c_{2}) U\left(\frac{\beta +bc_{1}c_{2}}{\tilde{N}}\right) \\ 
& =\sum_{\beta(c_{1}c_{2})}e\left(\frac{-\beta \overline{n_{1}}}{c_{1}} +\frac{\beta \overline{n_{2}}}{c_{2}} \right)\sum_{m}\int_\mathbb{R} I(\beta +uc_{1}c_{2},n_{1},c_{1})I(\beta +uc_{1}c_{2},n_{2},c_{2}) \\
 &  \hspace{3cm}\times   U\left(\frac{\beta +uc_{1}c_{2}}{\tilde{N}}\right)e(-mu)du.
\end{align*}
Substituting $v=\frac{\beta +uc_{1}c_{2}}{\tilde{N}}$, we obtain 
\begin{align}\label{S_6}
	S_{6}=\frac{\tilde{N}}{c_{1}c_{2}} \sum_{m}\mathcal{C}(m) \mathcal{J}(m),
\end{align}

where the character sum   $\mathcal{C}(m)=\sum_{\beta(c_{1}c_{2})}e\left(\frac{-\beta \overline{n_{1}}}{c_{1}} +\frac{\beta \overline{n_{2}}}{c_{2}}+\frac{m\beta}{c_{1}c_{2}} \right)$  and the integral 
\begin{align} \label{definition j m}
 \mathcal{J}(m)&:= \mathcal{J}(m; n_1, n_2, c_1, c_2) :=\int_{\mathbb{R}}I(v\tilde{N},n_{1},c_{1})I(v\tilde{N},n_{2},c_{2}) U\left(v\right)e(-mv)dv \notag\\
&=\iint_{\mathbb{R}^2} W(y_{1})W(y_{2})\left(\frac{y_{1}}{y_{2}}\right)^{it} e\left(\frac{-Nn_{1}y_{1}}{c_{1}} +\frac{Nn_{2}y_{2}}{c_{2}} \right) \notag \\
   & \hspace{2cm}\times  \left\lbrace \int_\mathbb{R}U(u)e\left(-\frac{\sqrt{N\tilde{N}y_{1}v}}{c_{1}}+\frac{\sqrt{N\tilde{N}y_{2}v}}{c_{2}}-\frac{m\tilde{N}v}{c_{1}c_{2}}\right)dv\right\rbrace dy_{1}dy_{2}.
\end{align} Integrating by parts $j$-times with respect to the variable $v$, we obtain

\begin{align*}
\mathcal{J}(m) &\ll_{j} \left(1+\frac{\sqrt{N\tilde{N}}}{c_{1}}+\frac{\sqrt{N\tilde{N}}}{c_{2}} \right)^{j} \left(\frac{c_{1}c_{2}}{m\tilde{N}} \right)^{j} \ \ \ \ll_{j}\left(\frac{N}{mK^{2}}\right).
\end{align*}
So the integral $\mathcal{J}(m)$ is negligibly small if $m\gg \frac{Nt^{\varepsilon}}{K^{2}}$. For $m=0$, using the bound  $I(m, n; c) \ll t^{-1/2}$, we obtain 
\begin{equation} \label{bound j 0}
\mathcal{J}(0) \ll t^{-1}. 
\end{equation} For $m \neq 0$,  changing the variables $y_1 = x_1^2$, $y_2 = x_2^2$, and $v = x_3^2$ in the equation \eqref{definition j m},  we obtain 
\begin{align*}
\mathcal{J}(m)&=\iiint_{\mathbb{R}^3} x_1 W(x_1) x_2 W(x_2) x_3 W(x_3) \exp( i G(x_1, x_2, x_3))dx_1dx_2dx_3, 
\end{align*} where 
\begin{align*} 
 G:=  
2 t \log x_1 - 2t \log x_2 - \frac{N n_1}{c_1} x_1^2 + \frac{N n_2}{c_2} x_2^2 - \frac{\sqrt{N \tilde{N}}}{c_1} x_1 x_3 + \frac{\sqrt{N \tilde{N}}}{c_2} x_2 x_3  - \frac{m \tilde{N}}{c_1 c_2} x_3^2. 
\end{align*} 
We apply \ref{exponential inte} in  $x_1$ variable first.
We have
\begin{align*}
\mathcal{J}(m)&=\iint_{\mathbb{R}^2}   W_1(x_2) W_2(x_3) \exp( i \tilde{G}( x_2, x_3))\int_{\mathbb{R}} W_1(x_1)\exp(iG_1(x_1))dx_1dx_2dx_3. 
\end{align*} 
where
\begin{align*}
\tilde{G}( x_2, x_3)= - 2t \log x_2 + \frac{N n_2}{c_2} x_2^2  + \frac{\sqrt{N \tilde{N}}}{c_2} x_2 x_3  - \frac{m \tilde{N}}{c_1 c_2} x_3^2. 
\end{align*}
and 
\begin{align*} 
 G_1(x_1) =  2 t \log x_1  - \frac{N n_1}{c_1} x_1^2  - \frac{\sqrt{N \tilde{N}}}{c_1} x_1 x_3 . 
\end{align*}
Let $x_1^0$ be the stationary point of $G_1(x_1)$.
Applying Lemma \ref{exponential inte}, we obtain: 
\begin{align*}
\mathcal{J}(m) \rightsquigarrow \iint_{\mathbb{R}^2}   W_2(x_2)  W_3(x_3) \exp( i \tilde{G}( x_2, x_3))\sqrt{2\pi}\frac{\exp(i\frac{\pi}{4}+ iG_1(x_1^0))}{\sqrt{|G_1^{(2)}(x_1^0)|}}dx_2dx_3. 
\end{align*} 
We note that 
\begin{align*}
    G_1^{(2)} (x_1)= -\frac{2 t}{x_1^2} \Rightarrow \left|G_1^{(2)}(x_1^0) \right| \asymp t.
\end{align*}
Similarly, applying the Lemma \ref{exponential inte}  in $x_2$ variable, we obtain:
\begin{align*}
\mathcal{J}(m)\rightsquigarrow\int_{\mathbb{R}}   W_3(x_3) \exp( i G_{3}( x_3))\sqrt{2\pi}\frac{\exp(i\frac{\pi}{4}+ iG_2(x_2^{0}))}{\sqrt{|G_2^{(2)}(x_2^{0})|}} \sqrt{2\pi}\frac{\exp(i\frac{\pi}{4}+ iG_1(x_1^{0}))}{\sqrt{|G_1^{(2)}(x_1^{0})|}}dx_3,
\end{align*} where 
\begin{align*}
G_3(x_3)=- \frac{m \tilde{N}}{c_1 c_2} x_3^2, \ \ \ \ 
G_2(x_2) =  -2 t \log x_2  + \frac{N n_2}{c_2} x_2^2  + \frac{\sqrt{N \tilde{N}}}{c_2} x_2 x_3 . 
\end{align*} 
 and 
 $x_{2}^{0} $ is the stationary point of $G_2(x_2)$.
 Like before, We note that $G_2^{(2)}(x_2^{0})$ is of size $t$ .
Thus, 
\begin{equation}\label{final j}
\mathcal{J}(m) \rightsquigarrow \int_{\mathbb{R}}   W_3(x_3) \exp( i G_{4}( x_3))\sqrt{2\pi}\frac{\exp(i\frac{\pi}{4})}{\sqrt{|G_2^{(2)}(x_2^{0})|}} \sqrt{2\pi}\frac{\exp(i\frac{\pi}{4})}{\sqrt{|G_1^{(2)}(x_1^{0})|}}dx_3.
\end{equation} 
where 
\begin{align*}
 G_{4}( x_3)=- \frac{m \tilde{N}}{c_1 c_2} x_3^2+G_2(x_2^{0})+G_1(x_1^{0}).
\end{align*}
Now we  apply the second derivative bound in  \eqref{final j}.
 We obtain that
 \begin{equation} \label{bound j m }
\mathcal{J}(m) \ll (tK)^{-1}. 
\end{equation}

We now consider the  character sum 
\begin{align*}
\mathcal{C}(m):=\sum_{\beta(c_{1}c_{2})}e\left(\frac{-\beta \overline{n_{1}}}{c_{1}} +\frac{\beta \overline{n_{2}}}{c_{2}}+\frac{m\beta}{c_{1}c_{2}} \right)  =c_{1}c_{2} \mathbbm{1}(\overline{n}_1 c_2 - \overline{n}_2 c_1 \equiv m (c_1 c_2)). 
\end{align*}

Substituting the evaluation of character sum in \eqref{S_6}, we obtain 
\begin{align*}
S_{6}=\tilde{N}\sum_{m\ll \frac{N}{K^{2}}} \mathbbm{1}(\overline{n}_1 c_2 - \overline{n}_2 c_1 \equiv m (c_1 c_2)) \mathcal{J}(m).
\end{align*}
Substituting the above estimate in \eqref{O 2 star}, we obtain
\begin{align*}
\mathcal{O}_2^\star =\tilde{N} \mathop{\sum \sum}_{c_{1},c_{2} \sim Q} \frac{1}{c_1 c_2}  \sum_{n_{1} \sim \frac{c_1 t}{N}} \sum_{n_{2} \sim \frac{c_2t}{N}} \sum_{m\sim \frac{N}{K^{2}}} \mathbbm{1}(\overline{n}_1 c_2 - \overline{n}_2 c_1 \equiv m (c_1 c_2)) \mathcal{J}(m).
\end{align*} Using  equation \eqref{bound j 0}, We observe that 
contribution of the diagonal term (when $c_1 = c_2$ and $n_1 = n_2$) is bounded from above by 

\begin{align} \label{diagonal}
\mathcal{O}_2^\star (d) =\tilde{N} \sum_{c \sim Q} \frac{1}{c^2}  \sum_{n \sim \frac{c t}{N}} |\mathcal{J}(0)| \ll \frac{ \tilde{N} }{N}. 
\end{align} 
Similarly using equation \eqref{bound j m }, contribution of the non-diagonal terms are bounded from above by 
\begin{align} \label{non diagonal}
\mathcal{O}_2^\star(nd) &=\tilde{N} \mathop{\sum \sum}_{c_{1},c_{2} \sim Q}  \sum_{n_1 \sim \frac{c_1 t}{N}} \sum_{n_{2} \sim \frac{c_2 t}{N}} \frac{1}{c_1 c_2}\sum_{m\sim \frac{N}{K^2}} \frac{1}{c_1 c_2} |\mathcal{J}(m)| \notag \\
& \ll \tilde{N} \mathop{\sum \sum}_{c_{1},c_{2} \sim Q} \times \frac{c_1 t}{N} \frac{c_2 t}{N} \frac{1}{c_1 c_2} \frac{N}{K^2} \frac{1}{c_1 c_2} \frac{1}{t K} \ll \frac{\tilde{N} t }{NK^3}. 
\end{align} Using  bounds of equations \eqref{diagonal} and \eqref{non diagonal} in equation \eqref{definition O 2 star},  we obtain

\begin{align*} 
\mathcal{O^{*}} \ll  N^{3/2}   \left(\frac{\tilde{N} }{N} + \frac{\tilde{N} t }{NK^3}\right)^\frac{1}{2} \ll  N \sqrt{\tilde{N}}  \left(1   + \frac{\sqrt{t}}{K^{3/2}}\right) \ll \sqrt{N} Q K^2 \left(1   + \frac{\sqrt{t}}{K^{3/2}}\right). 
\end{align*} This prove our Proposition \ref{proposition o star}.

 {\bf Acknowledgement:} Authors are grateful to Prof. Ritabrata Munshi as most of this paper is based on the ideas that he shared with us. Authors would also like to thank Prof. Satadal Ganguly for useful suggestions and comments.  Author would also like to thank  Stat-Math unit, Indian Statistical Institute, Kolkata for the wonderful academic atmosphere. During the work,  S. Singh was supported by the Department of Atomic Energy, Government of India, NBHM post doctoral fellowship no: 2/40(15)/2016/R$\&$D-II/5765.

{}
\end{document}